\documentclass[12pt,a4paper]{amsart}
\usepackage{graphicx}
\usepackage{latexsym}
\usepackage{amsxtra}
\usepackage{amsmath}
\usepackage{amsfonts}
\usepackage{amssymb}

\newtheorem{theorem}{Theorem}[section]
\newtheorem{corollary}[theorem]{Corollary}
\newtheorem{lemma}[theorem]{Lemma}

\usepackage[figuresright]{rotating}

\title{Densities of the Raney distributions}
\author{Wojciech M{\l}otkowski, Karol A. Penson, Karol \.{Z}yczkowski}
\thanks{
K.~\.{Z}. is supported  by the Grant DEC-2011/02/A/ST1/00119
of Polish National Centre of Science.
K.~A.~P. acknowledges support from PAN/CNRS  under Project PICS No.
4339 and from Agence Nationale de la Recherche (Paris, France) under
Program PHYSCOMB No. ANR-08-BLAN-0243-2.}
\address{Instytut Matematyczny,
Uniwersytet Wroc{\l}awski,
Plac~Grunwaldzki~2/4,
50-384 Wroc{\l}aw, Poland}
\email{mlotkow@math.uni.wroc.pl}
\address{Laboratoire de Physique Th\'{e}orique de la Mati\`{e}re
Condens\'{e}e (LPTMC), Universit\'{e} Pierre et Marie Curie, CNRS UMR
7600, Tour 13 - 5i\`{e}me \'{e}t., Bo\^{i}te Courrier 121, 4 place
Jussieu, F 75252 Paris Cedex 05, France}
\email{penson@lptl.jussieu.fr}
\address{Institute of Physics, Jagiellonian University, Cracow
and Center for Theoretical Physics, Polish Academy of Sciences,
Warsaw, Poland}
\email{karol@tatry.if.uj.edu.pl}
\begin{document}

\begin{abstract}
We prove that if $p\ge1$ and $0< r\le p$ then the sequence $\binom{mp+r}{m}\frac{r}{mp+r}$,
$m=0,1,2,\ldots$, is positive definite,
more precisely, is the moment sequence of a probability measure
$\mu(p,r)$ with compact support contained in $[0,+\infty)$.
This family of measures encompasses the
multiplicative free powers of the Marchenko-Pastur distribution
as well as the Wigner's semicircle distribution centered at $x=2$.
We show that if $p>1$ is a rational number, $0<r\le p$, then $\mu(p,r)$
is absolutely continuous and its density $W_{p,r}(x)$ can be expressed
in terms of the Meijer and the generalized hypergeometric functions. In some cases,
including the multiplicative free square and the multiplicative free square root
of the Marchenko-Pastur measure, $W_{p,r}(x)$ turns out to be an elementary function.
\end{abstract}

\maketitle
\today

\section*{Introduction}

For $p,r\in\mathbb{R}$ we define the \textit{Raney numbers}
(or \textit{two-parameter Fuss-Catalan numbers}) by
\begin{equation}\label{intraneysequence}
A_{m}(p,r):=
\frac{r}{m!}\prod_{i=1}^{m-1}(mp+r-i),
\end{equation}
$A_0(p,r):=1$. For $m=0,1,2,\ldots$ we can also write
\begin{equation}
A_{m}(p,r)=\binom{mp+r}{m}\frac{r}{mp+r},
\end{equation}
(provided  $mp+r\ne0$) where the generalized binomial is defined by
\[\binom{a}{m}:=\frac{a(a-1)\ldots(a-m+1)}{m!}.\]
Let $\mathcal{B}_{p}(z)$ denote the generating function of the sequence
$\left\{A_{m}(p,1)\right\}_{m=0}^{\infty}$, the \textit{Fuss numbers of order $p$}:
\begin{equation}
\mathcal{B}_p(z):=\sum_{m=0}^{\infty}A_{m}(p,1)z^m,
\end{equation}
convergent in some neighborhood of $0$.
For example
\begin{equation}
\mathcal{B}_2(z)=\frac{2}{1+\sqrt{1-4z}}.
\end{equation}
Lambert showed that
\begin{equation}
\mathcal{B}_p(z)^r=\sum_{m=0}^{\infty}A_{m}(p,r)z^m,
\end{equation}
see \cite{gkp}.
These generating functions also satisfy
\begin{equation}
\mathcal{B}_{p}(z)=1+z\mathcal{B}_{p}(z)^{p},
\end{equation}
which reflects the identity $A_{m}(p,p)=A_{m+1}(p,1)$, and
\begin{equation}
\mathcal{B}_p(z)=\mathcal{B}_{p-r}\big(z\mathcal{B}_p(z)^r\big).
\end{equation}

It was shown in \cite{mlotkowski2010} that if $p\ge1$ and $0\le r\le p$ then the sequence
$\left\{A_{m}(p,r)\right\}_{m=0}^{\infty}$ is positive definite,
i.e. is the moment sequence of a probability measure $\mu(p,r)$ on $\mathbb{R}$.
Moreover, $\mu(p,r)$ has compact support (and therefore is unique)
contained in the positive half-line $[0,\infty)$
(for example $\mu(p,0)=\delta_{0}$).
The proof involved methods from the free probability theory (see \cite{vdn,ns,biane}).
In particular, for $p\ge1$
\begin{equation}
\mu(p,1)=\mu(2,1)^{\boxtimes(p-1)},
\end{equation}
where $\boxtimes$ denotes the multiplicative free power,
and $\mu(2,1)$ is known as the \textit{Marchenko-Pastur}
(called also the \textit{free Poisson}) \textit{distribution}. It is given by
\begin{equation}
\mu(2,1)=\frac{1}{2\pi}\sqrt{\frac{4-x}{x}}\,dx\qquad\hbox{on [0,4],}
\end{equation}
and plays an important role in the theory of random matrices,
see \cite{wegmann1976,haageruplarsen2000,haagerup2005,anderson2010,alexeev2010,bbcc}.
It was proved in \cite{alexeev2010} that the measure $\mu(2,1)^{\boxtimes n}=\mu(n+1,1)$
is the limit of the distribution of squared singular values of the power $G^n$
of a random matrix $G$, when the size of the matrix $G$ goes to infinity.

In this paper we are going to prove positive definiteness
of $\left\{A_{m}(p,r)\right\}_{m=0}^{\infty}$ using more classical methods.
Namely, we show that if $p>1$, $0<r\le p$ and if $p$ is a rational number then
$\mu(p,r)$ is absolutely continuous and can be represented
as Mellin convolution of modified beta measures.
Next we provide a formula for the density $W_{p,r}(x)$
of $\mu(p,r)$ in terms of the Meijer function
and consequently, of the generalized hypergeometric functions
(cf. \cite{zpnc2011,pezy}, where $p$ was assumed to be an integer).
This allows us to draw graphs of these densities
and, in some particular cases, to express $W_{p,r}(x)$
as an elementary function.
It is worth to point out that for $r=1$ an alternative description of
the densities $W_{p,1}(x)$ has been recently given by Haagerup
and M\"{o}ller, see Corollary~3 in \cite{haagerupmoller2012}.

Finally let us also mention that
the measures $\mu(p,r)$ satisfy a peculiar relation:
\begin{equation}
\mu(p,r)\triangleright\mu(p+s,s)=\mu(p+s,r+s),
\end{equation}
for $p\ge1$, $0<r\le p$ and $s>0$,
see \cite{mlotkowski2010}, involving monotonic convolution ``$\triangleright$",
an associative, noncommutative operation on probability measures
on $\mathbb{R}$, introduced by Muraki \cite{muraki}.

\section{Preliminaries}

For probability measures $\mu_1$, $\mu_2$
on the positive half-line $[0,\infty)$ the \textit{Mellin convolution} is defined by
\begin{equation}
\left(\mu_1\circ\mu_2\right)(A):=\int_{0}^{\infty}\int_{0}^{\infty}\mathbf{1}_{A}(xy)d\mu_1(x)d\mu_{2}(y)
\end{equation}
for every Borel set $A\subseteq[0,\infty)$.
This is the distribution of product $X_1\cdot X_2$ of two independent nonnegative
random variables with $X_i\sim\mu_i$.
In particular,
if $c>0$ then $\mu\circ\delta_c$ is the \textit{dilation} of $\mu$:
\[
\left(\mu\circ\delta_c\right)(A)=\mathbf{D}_c\mu(A):=\mu\left(\frac{1}{c}A\right).
\]
Note that if $\mu$ has density $f(x)$
then $\mathbf{D}_c(\mu)$ has density $f(x/c)/c$.

If both the measures $\mu_1,\mu_2$ have all \textit{moments}
\[
s_m(\mu_i):=\int_{0}^{\infty}x^m\,d\mu_i(x)
\]
finite then so has $\mu_1\circ\mu_2$ and
\[
s_m\left(\mu_1\circ\mu_2\right)=s_m(\mu_1)\cdot s_m(\mu_2)
\]
for all $m$.

If $\mu_1,\mu_2$ are absolutely continuous, with densities $f_1,f_2$ respectively, then
so is $\mu_1\circ\mu_2$ and its density is given by the Mellin convolution:
\[
\left(f_1\circ f_2\right)(x):=\int_{0}^{\infty}f_{1}(x/y)f_{2}(y)\frac{dy}{y}.
\]

We will need the following \textit{modified beta distributions}:

\begin{lemma}\label{bprelembetadistr}
Let $u,v,l>0$. Then
\[
\left\{\frac{\Gamma(u+n/l)\Gamma(u+v)}{\Gamma(u+v+n/l)\Gamma(u)}\right\}_{n=0}^{\infty}
\]
is the moment sequence of the probability measure
\begin{equation}\label{bprebetadist}
\mathbf{b}(u+v,u,l):=
\frac{l}{\mathrm{B}(u,v)}x^{lu-1}\left(1-x^l\right)^{v-1}dx
\end{equation}
on $[0,1]$, where $\mathrm{B}$ is the Euler beta function.
\end{lemma}

\begin{proof}
Using the substitution $t=x^l$ we obtain:
\begin{align*}
\frac{\Gamma(u+n/l)\Gamma(u+v)}{\Gamma(u+v+n/l)\Gamma(u)}
&=\frac{\mathrm{B}(u+n/l,v)}{\mathrm{B}(u,v)}=
\frac{1}{\mathrm{B}(u,v)}\int_{0}^{1}t^{u+n/l-1}(1-t)^{v-1}dt\\
&=\frac{l}{\mathrm{B}(u,v)}\int_{0}^{1}x^{lu+n-1}\left(1-x^l\right)^{v-1}dx.
\end{align*}
\end{proof}

Note that if $X$ is a positive random variable whose distribution
has density $f(x)$ and if $l>0$ then the distribution of $X^{1/l}$
has density $lx^{l-1}f(x^l)$.
In particular, if $X$ has \textit{beta distribution} $\mathbf{b}(u+v,u,1)$
then $X^{1/l}$ has distribution $\mathbf{b}(u+v,u,l)$.

For $u,l>0$ we also define
\begin{equation}\label{bprebetadistb}
\mathbf{b}(u,u,l):=\delta_{1}.
\end{equation}

\section{Applying Mellin convolution}

From now on we assume that $p>1$ is a rational number,
say $p=k/l$, with $1\le l<k$, and that $0<r\le p$.
We will show, that then the sequence $A_m(p,r)$
is a moment sequence of a probability measure $\mu(p,r)$, which
can be represented as Mellin convolution of modified beta distributions.
In particular, $\mu(p,r)$ is absolutely continuous and we will denote
its density by $W_{p,r}$.
The case when $p$ is an integer was studied in \cite{pezy,zpnc2011}.

First we need to express the numbers $A_{m}(p,r)$ in a special form.

\begin{lemma}\label{convlemformula}
If $p=k/l$, where $k,l$ are integers, $1\le l<k$ and $0<r\le p$ then
\begin{equation}\label{convformula}
A_{m}(p,r)=\frac{r}{\sqrt{2\pi k l(k-l)}}\left(\frac{p}{p-1}\right)^{r}
\frac{\prod_{j=1}^{k}\Gamma(\beta_j+m/l)}{\prod_{j=1}^{k}\Gamma(\alpha_j+m/l)}
c(p)^{m},
\end{equation}
where $c(p)=p^p(p-1)^{1-p}$,
\begin{align}
\alpha_j&=\left\{\begin{array}{ll}
\cfrac{j}{l}&\mbox{if $1\le j\le l$,}\label{convalpha}\\
\cfrac{r+j-l}{k-l}&\mbox{if $l+1\le j\le k$,}
\end{array}\right.\\
\beta_j&=\frac{r+j-1}{k},\quad\qquad{1\le j\le k}.\label{convbeta}
\end{align}
\end{lemma}

\begin{proof}
First we write:
\begin{equation}\label{convgamma}
\binom{mp+r}{m}\frac{r}{mp+r}=\frac{r\Gamma(mp+r)}{\Gamma(m+1)\Gamma(mp-m+r+1)}.
\end{equation}
Now we apply the Gauss's multiplication formula:
\[
%\begin{equation}\label{gammagauss}
\Gamma(nz)=(2\pi)^{(1-n)/2}n^{nz-1/2}\Gamma(z)\Gamma\left(z+\frac{1}{n}\right)
\Gamma\left(z+\frac{2}{n}\right)\ldots\Gamma\left(z+\frac{n-1}{n}\right)
%\end{equation}
\]
to get:
\begin{align*}
\Gamma(m p+r)&=\Gamma\left(k\left(\frac{m}{l}+\frac{r}{k}\right)\right)%\\&
=(2\pi)^{(1-k)/2}k^{m k/l+r-1/2}\prod_{j=1}^{k}
\Gamma\left(\frac{m}{l}+\frac{r+j-1}{k}\right),\\
\Gamma(m+1)&=\Gamma\left(l\frac{m+1}{l}\right)=(2\pi)^{(1-l)/2}l^{m+1/2}\prod_{j=1}^{l}
\Gamma\left(\frac{m}{l}+\frac{j}{l}\right)
\end{align*}
and
\begin{align*}
\Gamma(m p-m+r+1)&=\Gamma\left((k-l)\left(\frac{m}{l}+\frac{r+1}{k-l}\right)\right)\\
&=(2\pi)^{(1-k+l)/2}(k-l)^{m(k-l)/l+r+1/2}\prod_{j=l+1}^{k}
\Gamma\left(\frac{m}{l}+\frac{r+j-l}{k-l}\right).
\end{align*}
It remains to apply them to (\ref{convgamma}).
\end{proof}

Now we need to modify enumeration of $\alpha$'s.

\vfill\eject

\begin{lemma}\label{convlemalpha}
For $1\le i\le l+1$ denote
\[
j_i:=\left\lfloor\frac{(i-1)k}{l}\right\rfloor+1,
\]
where $\lfloor\cdot\rfloor$ is the floor function,
so that
\[
1=j_1<j_2<\ldots<j_{l}<k<k+1=j_{l+1}.
\]
For $1\le j\le k$ define
\begin{equation}\label{convalphatilde}
\widetilde{\alpha}_j=\left\{\begin{array}{ll}
\cfrac{i}{l}&\mbox{ if $j=j_i$, $1\le i\le l$,}\\
\cfrac{r+j-i}{k-l}&\mbox{ if $j_i<j<j_{i+1}$.}
\end{array}\right.
\end{equation}
Then the sequence $\left\{\widetilde{\alpha}_j\right\}_{j=1}^{k}$
is a rearrangement of $\left\{{\alpha}_j\right\}_{j=1}^{k}$.
Moreover, if $0<r\le p=k/l$ then we have $\beta_j\le\widetilde{\alpha}_j$
for all $j\le k$.
\end{lemma}

\begin{proof}
It is easy to verify the first statement.

Assume that $j=j_i$ for some $i\le l$.
Then we have to show that
\[
\frac{r+j_{i}-1}{k}\le\frac{i}{l},
\]
which is equivalent to
\[
lr+l\left\lfloor\frac{k(i-1)}{l}\right\rfloor\le ki,
\]
and the latter is a consequence of the fact that $\lfloor x\rfloor\le x$
and the assumption $r\le p=k/l$.

Now assume that $j_{i}<j<j_{i+1}$. We ought to show that
\[
\frac{r+j-1}{k}\le\frac{r+j-i}{k-l},
\]
which is equivalent to
\[
lr+lj+k-l-ki\ge0.
\]
Using the inequality $\lfloor x\rfloor+1>x$ we obtain
\[
lj+k-l-ki\ge l(j_i+1)+k-l-ki=lj_{i}+k-ki>k(i-1)+k-ki=0,
\]
which completes the proof, as $r>0$.
\end{proof}

Now we are ready to prove the main theorem of this section.

\begin{theorem}
Suppose that $p=k/l$, where $k,l$ are integers such that
$1\le l<k$, and that $r$ is a real number, $0<r\le p$.
Then there exists a unique probability measure $\mu(p,r)$
such that (\ref{intraneysequence}) is its moment sequence. Moreover
$\mu(p,r)$ can be represented as the following Mellin
convolution:
\[
\mu(p,r)=\mathbf{b}(\widetilde{\alpha}_1,\beta_1,l)\circ\ldots\circ
\mathbf{b}(\widetilde{\alpha}_k,\beta_k,l)\circ\delta_{c(p)},
\]
where \[c(p):=\frac{p^p}{(p-1)^{p-1}}.\]
Consequently, $\mu(p,r)$ is absolutely continuous
and its support is $[0,c(p)]$.
\end{theorem}

Note that the representation of densities in the form of Mellin convolution
of modified beta distributions was used in different context in \cite{gorska2012},
see its Appendix~A.

\textbf{Example.} For the Marchenko-Pastur measure we get
the following decomposition:
\begin{equation}
\mu(2,1)=\mathbf{b}(1,1/2,1)\circ\mathbf{b}(2,1,1)\circ\delta_{4},
\end{equation}
where $\mathbf{b}(1,1/2,1)$ has density $1/(\pi\sqrt{x-x^2})$ on $[0,1]$,
the arcsine distribution with the moment sequence $\binom{2m}{m}4^{-m}$,
and $\mathbf{b}(2,1,1)$ is the Lebesgue measure on $[0,1]$
with the moment sequence $1/(m+1)$.

\begin{proof}
In view of Lemma~\ref{convlemformula} and Lemma~\ref{convlemalpha}
we can write
\[
A_m(p,r)=D
\prod_{j=1}^{k}\frac{\Gamma(\beta_j+m/l)\Gamma(\widetilde{\alpha}_j)}
{\Gamma(\widetilde{\alpha}_j+m/l)\Gamma(\beta_j)}\cdot c(p)^m
\]
for some constant $D$. Taking $m=0$ we see that $D=1$.
\end{proof}

Note that a part of the theorem illustrates a result of Kargin \cite{kargin2008},
who proved that if $\mu$ is a compactly supported probability measure
on $[0,\infty)$, with expectation $1$ and variance $V$,
and if $L_n$ denotes the supremum of the support
of the multiplicative free convolution power $\mu^{\boxtimes n}$,
then
\begin{equation}\label{convkargin}
\lim_{n\to\infty}\frac{L_n}{n}=eV,
\end{equation}
where $e=2.71\ldots$ is the Euler's number.
The Marchenko-Pastur measure $\mu(2,1)$ has expectation and variance equal to $1$
and $\mu(2,1)^{\boxtimes n}=\mu(n+1,1)$, so in this case $L_n=(n+1)^{n+1}/n^n$
(this was also proved in \cite{wegmann1976} and \cite{haagerup2005})
and (\ref{convkargin}) holds.

The density function for $\mu(p,r)$ will be denoted by $W_{p,r}(x)$.
Since $A_{m}(p,p)=A_{m+1}(p,1)$, we have
\begin{equation}\label{convpp}
W_{p,p}(x)=x\cdot W_{p,1}(x),
\end{equation}
for example
\begin{equation}\label{convsemicircle}
W_{2,2}(x)=\frac{1}{2\pi}\sqrt{x(4-x)}\qquad\hbox{on $[0,4]$,}
\end{equation}
which is the famous semicircle Wigner distribution with radius $2$, centered at $x=2$.

Now we can reprove the main result of \cite{mlotkowski2010}.

\begin{theorem}
Suppose that $p,r$ are real numbers such that $p\ge1$ and $0\le r\le p$.
Then there exists a unique probability measure $\mu(p,r)$,
with support contained in $[0,c(p)]$, such that
$\left\{A_{m}(p,r)\right\}_{m=0}^{\infty}$ is its moment sequence.
\end{theorem}

\begin{proof}
It follows from the fact that the class of positive definite sequence
is closed under pointwise limits.
\end{proof}

\section{Applying Meijer $G$-function}

The aim of this section is to describe the density function
$W_{p,r}(x)$ of $\mu(p,r)$ in terms of the Meijer $G$-function (see \cite{olver} for example)
and consequently, as a linear combination of generalized hypergeometric
functions. We will see that in some particular cases $W_{p,r}$
can be represented as an elementary function.

\begin{theorem}
Let $p=k/l>1$, where $k,l$ are integers such that
$1\le l<k$, and let $r$ be a positive real number, $r\le p$.
Then the density $W_{p,r}$ of the probability measure $\mu(p,r)$
can be expressed as
\begin{equation}\label{meijerth}
W_{p,r}(x)=\frac{rp^{r}}{x(p-1)^{r+1/2}\sqrt{2k\pi}}\,
G^{k,0}_{k,k}\!\left(\left.\!\!
\begin{array}{ccc}
\alpha_1,\!\!&\!\!\ldots,\!\!&\!\!\alpha_k\\
\beta_1,\!\!&\!\!\ldots,\!\!&\!\!\beta_k
\end{array}
\!\right|\frac{x^{l}}{c(p)^l}\right),
\end{equation}
$x\in(0,c(p))$, where $c(p)=p^p(p-1)^{1-p}$
and the parameters $\alpha_j,\beta_j$ are given by
(\ref{convalpha}) and (\ref{convbeta}).
\end{theorem}

\begin{proof}
 Define
\[
\phi_{p,r}(\sigma)=\frac{r\Gamma(\sigma p-p+r)}{\Gamma(\sigma)\Gamma(\sigma p-\sigma-p+r+2)}.
\]
If $m$ is a natural number then
\[
\phi_{p,r}(m+1)=\binom{mp+r}{m}\frac{r}{mp+r}
\]
so $\phi_{p,r}$ is the Mellin transform of the density function $W_{p,r}$
of $\mu(p,r)$:
\[
\phi_{p,r}(\sigma)=\int_{0}^{\infty} x^{\sigma-1} W_{p,r}(x)\,dx.
\]

In order to reconstruct $W_{p,r}$ we apply the inverse Mellin transform:
\[
W_{p,r}(x)=\frac{1}{2\pi\mathrm{i}}\int_{C-\mathrm{i}\infty}^{C+\mathrm{i}\infty}
x^{-\sigma}\phi_{p,r}(\sigma)\,d\sigma,
\]
see \cite{andrews,olver,polyanin} for details.
Putting $m=\sigma-1$ in (\ref{convformula}) we get
\[
\phi_{p,r}(\sigma)=\frac{r(p-1)^{p-1-r}}{p^{p-r}\sqrt{2\pi k l(k-l)}}
\frac{\prod_{j=1}^{k}\Gamma(\beta_j-1/l+\sigma/l)}{\prod_{j=1}^{k}\Gamma(\alpha_j-1/l+\sigma/l)}
c(p)^{\sigma}.
\]
Writing the right hand side as $\Phi(\sigma/l)c(p)^{\sigma}$,
using the substitution $\sigma=lu$ and the definition of the Meijer $G$-function
(see \cite{olver} for example) we obtain
\[
W_{p,r}(x)=\frac{1}{2\pi \mathrm{i}}\int_{C-\mathrm{i}\infty}^{C+\mathrm{i}\infty}
\Phi(\sigma/l)c(p)^{\sigma}x^{-\sigma}d\sigma
=\frac{l}{2\pi \mathrm{i}}\int_{C-\mathrm{i}\infty}^{C+\mathrm{i}\infty}
\Phi(u)\left({x^l}/{c(p)^{l}}\right)^{-u}du
\]
\[
=\frac{r(p-1)^{p-r-3/2}}{p^{p-r}\sqrt{2\pi k }}
G^{k,0}_{k,k}\!\left(\left.\!\!
\begin{array}{ccc}
\alpha_1^{-},\!\!&\!\!\ldots,\!\!&\!\!\alpha_k^{-}\\
\beta_1^{-},\!\!&\!\!\ldots,\!\!&\!\!\beta_k^{-}
\end{array}
\!\right|z\right),
\]
where $z=x^l/c(p)^l$, $\alpha_j^{-}=\alpha_j-1/l$, $\beta_j^{-}=\beta_j-1/l$.
Finally we use formula (16.19.2) in \cite{olver} and obtain
\begin{equation}
W_{p,r}(x)=\frac{r(p-1)^{p-r-3/2}}{z^{1/l}p^{p-r}\sqrt{2\pi k }}
G^{k,0}_{k,k}\!\left(\left.\!\!
\begin{array}{ccc}
\alpha_1,\!\!&\!\!\ldots,\!\!&\!\!\alpha_k\\
\beta_1,\!\!&\!\!\ldots,\!\!&\!\!\beta_k
\end{array}
\!\right|z\right),
\end{equation}
which is equivalent to (\ref{meijerth}).
\end{proof}

Now applying Slater's theorem (see (16.17.2) in \cite{olver}) we can represent $W_{p,r}$ as a linear
combination of hypergeometric functions.

\begin{theorem}\label{thmslater}
For for $p=k/l$, with $1\le l<k$, $r>0$ and $x\in(0,c(p))$ we have
\begin{equation}\label{meislater}
W_{p,r}(x)=\gamma(k,l,r)\sum_{h=1}^{k}c(h,k,l,r)\,
{}_{k}F_{k-1}\!\left(\left.
\begin{array}{c}
\!\!\!\mathbf{a}(h,k,l,r)\\
\!\!\!\mathbf{b}(h,k,l,r)\end{array}\!
\right|z\right)
z^{(r+h-1)/k-1/l}
\end{equation}
where $z=x^l/c(p)^l$,
\begin{align}
\gamma(k,l,r)&=\frac{r(p-1)^{p-r-3/2}}{p^{p-r}\sqrt{2\pi k }},\\
c(h,k,l,r)&=\frac{\prod_{j=1}^{h-1}\Gamma\left(\frac{j-h}{k}\right)
\prod_{j=h+1}^{k}\Gamma\left(\frac{j-h}{k}\right)}
{\prod_{j=1}^{l}\Gamma\left(\frac{j}{l}-\frac{r+h-1}{k}\right)
\prod_{j=l+1}^{k}\Gamma\left(\frac{r+j-l}{k-l}-\frac{r+h-1}{k}\right)},
\end{align}
and the parameter vectors of the hypergeometric functions are
\begin{align}
\mathbf{a}(h,k,l,r)&=
\left(\left\{\frac{r+h-1}{k}-\frac{j-l}{l}\right\}_{j=1}^{l},
\left\{\frac{r+h-1}{k}-\frac{r+j-k}{k-l}\right\}_{j=l+1}^{k}\right),\\
\mathbf{b}(h,k,l,r)&=
\left(\left\{\frac{k+h-j}{k}\right\}_{j=1}^{h-1},
\left\{\frac{k+h-j}{k}\right\}_{j=h+1}^{k}\right).
\end{align}
\end{theorem}

The most tractable case is $p=2$.

\medskip
\begin{corollary}\label{cordwa}
For $p=2$, $0<r\le 2$, the density function is
\begin{align}
W_{2,r}(x)%=\sin\left(\pi r/2-r\arcsin\sqrt{x/4}\right)\frac{x^{r/2-1}}{\pi}
&=\frac{\sin\left(r\cdot\arccos\sqrt{x/4}\right)}{\pi x^{1-r/2}},\\
\intertext{$x\in(0,4)$. In particular for $r=1/2$ we have}
W_{2,1/2}(x)&=\frac{\sqrt{2-\sqrt{x}}}{2\pi x^{3/4}},\\
\intertext{and for $r=3/2$}
W_{2,3/2}(x)&=\frac{\left(\sqrt{x}+1\right)\sqrt{2-\sqrt{x}}}{2\pi x^{1/4}}.
\end{align}
\end{corollary}

Note that if $r>2$ then $W_{2,r}(x)<0$ for some values of $x\in(0,4)$.

\begin{proof} We take $k=2$, $l=1$ so that $c(2)=4$, $z=x/4$ and
$\gamma(2,1,r)={r2^r}/{\left(8\sqrt{\pi}\right)}.$
Using the Euler's reflection formula and the identity $\Gamma(1+r/2)=\Gamma(r/2)r/2$ we get
\begin{align*}
c(1,2,1,r)&=\frac{\Gamma(1/2)}{\Gamma(1-r/2)\Gamma(1+r/2)}
=\frac{2\sin(\pi r/2)}{r\sqrt{\pi}},\\
c(2,2,1,r)&=\frac{\Gamma(-1/2)}{\Gamma((1-r)/2)\Gamma((1+r)/2)}=\frac{-2\cos(\pi r/2)}{\sqrt{\pi}}.
\end{align*}
We also need formulas for two hypergeometric functions, namely
\begin{align*}
{}_{2}F_{1}\!\left(\left.
\frac{r}{2},\frac{-r}{2};\,\frac{1}{2}\,
\right|z\right)&=\cos(r\arcsin\sqrt{z}),\\
{}_{2}F_{1}\!\left(\left.
\frac{1+r}{2},\frac{1-r}{2};\,\frac{3}{2}\,
\right|z\right)&=\frac{\sin(r\arcsin\sqrt{z})}{r\sqrt{z}},
\end{align*}
see 15.4.12 and 15.4.16 in \cite{olver}.
Now we can write
\[
W_{2,r}(x)=\frac{\sin(\pi r/2)\cos\left(r\arcsin\sqrt{x/4}\right)%\frac{x^{r/2-1}}{\pi}
-\cos(\pi r/2)\sin\left(r\arcsin\sqrt{x/4}\right)}{\pi x^{1-r/2}}
\]
\[
=\frac{\sin\left(\pi r/2-r\arcsin\sqrt{x/4}\right)}{\pi x^{1-r/2}}
=\frac{\sin\left(r\arccos\sqrt{x/4}\right)}{\pi x^{1-r/2}},
\]
which concludes the proof.
\end{proof}

\textbf{Remark.}
Note that
\begin{equation}
\frac{W_{2,1}\left(\sqrt{x}\right)}{2\sqrt{x}}=\frac{1}{4}W_{2,1/2}\left(\frac{x}{4}\right).
\end{equation}
It means that if $X,Y$ are random variables such that $X\sim\mu(2,1)$ and $Y\sim\mu(2,1/2)$
then $X^2\sim4Y$. This can be also derived from the relation $A_{m}(2,1/2)4^m=A_{2m}(2,1)$.

\section{Some particular cases}

In this part we will see that for $k=3$ some densities still can
be represented as elementary functions. We will need two
families of
% hypergeometric
formulas (cf. 15.4.17 in \cite{olver}).

\begin{lemma}
For $c\ne0,-1,-2,\ldots$ we have
\begin{align}
{}_{2}F_{1}\!\left(\left.
\frac{c}{2},\frac{c-1}{2};\,c\right|z\right)
&=2^{c-1}\big(1+\sqrt{1-z}\big)^{1-c},\label{parthyperg1}\\
{}_{2}F_{1}\!\left(\left.
\frac{c+1}{2},\frac{c-2}{2};\,c\right|z\right)
&=\frac{2^{c-1}}{c}\big(1+\sqrt{1-z}\big)^{1-c}\big(c-1+\sqrt{1-z}\big).\label{parthyperg2}
\end{align}
\end{lemma}

\begin{proof}
We know that ${}_{2}F_{1}\!\left(\left.a,b;\,c\right|z\right)$
is the unique function $f$ which is analytic at $z=0$,
with $f(0)=1$, and satisfies the \textit{hypergeometric equation}:
\[
z(1-z)f''(z)+\big[c-(a+b+1)z\big]f'(z)-abf(z)=0
\]
(see \cite{andrews}).
Now one can check that this equation is satisfied
by the right hand sides of (\ref{parthyperg1}) and (\ref{parthyperg2})
for given parameters~$a,b,c$.
\end{proof}

Now consider $p=3/2$.

\begin{theorem}\label{thmtrzydrugie}
Assume that $p=3/2$. Then for $r=1/2$ we have
\begin{equation}\label{partpoltorapol}
W_{3/2,1/2}(x)=\frac{\left(1+\sqrt{1-4x^2/27}\right)^{2/3}-\left(1-\sqrt{1-4x^2/27}\right)^{2/3}}
{2^{5/3}3^{-1/2}\pi x^{2/3}},
\end{equation}
for $r=1$:
\begin{equation}\label{partpoltorajeden}
W_{3/2,1}(x)=3^{1/2}\frac{\left(1+\sqrt{1-4x^2/27}\right)^{1/3}
-\left(1-\sqrt{1-4x^2/27}\right)^{1/3}}
{2^{4/3}\pi x^{1/3}}\end{equation}
\[+3^{1/2}x^{1/3}\frac{\left(1+\sqrt{1-4x^2/27}\right)^{2/3}
-\left(1-\sqrt{1-4x^2/27}\right)^{2/3}}
{2^{5/3}\pi}
\]
and, finally, $W_{3/2,3/2}(x)=x\cdot W_{3/2,1}(x)$,
with $x\in(0,3\sqrt{3}/2)$.
\end{theorem}

\begin{proof} For arbitrary $r$ we have
\begin{align*}
W_{3/2,r}(x)=\frac{2^{1-2r/3}\sin\big(2\pi r/3\big)}{3^{3/2-r}\pi}\,
{}_{3}F_{2}\!\left(\left.
\frac{3+2r}{6},\frac{r}{3},\frac{-2r}{3};\,\frac{2}{3},\frac{1}{3}
\right|z\right)&z^{r/3-1/2}\\
-\frac{2^{(4-2r)/3}r\sin\big((1-2r)\pi/3\big)}{3^{3/2-r}\pi}\,
{}_{3}F_{2}\!\left(\left.
\frac{5+2r}{6},\frac{1+r}{3},\frac{1-2r}{3};\,\frac{4}{3},\frac{2}{3}
\right|z\right)&z^{(r+1)/3-1/2}\\
-\frac{r(1+2r)\sin\big((1+2r)\pi/3\big)}{2^{(1+2r)/3}3^{3/2-r}\pi}\,
{}_{3}F_{2}\!\left(\left.
\frac{7+2r}{6},\frac{2+r}{3},\frac{2-2r}{3};\,\frac{5}{3},\frac{4}{3}
\right|z\right)&z^{(r+2)/3-1/2},
\end{align*}
where $z=4x^2/27$. If $r=1/2$ or $r=1$ then one term vanishes
and in the two others the hypergeometric functions reduce to ${}_2 F_{1}$.

For $r=1/2$ we apply (\ref{parthyperg1}) to obtain:
\[
W_{3/2,1/2}(x)=\frac{z^{-1/3}}{2^{1/3}3^{1/2}\pi}\,
{}_{2}F_{1}\!\left(\left.
\frac{1}{6},\frac{-1}{3};\,\frac{1}{3}\right|z\right)
-\frac{z^{1/3}}{2^{5/3}3^{1/2}\pi}\,
{}_{2}F_{1}\!\left(\left.
\frac{5}{6},\frac{1}{3};\,\frac{5}{3}\right|z\right)\]
\[
=\frac{z^{-1/3}}{2^{1/3}3^{1/2}\pi}\,
2^{-2/3}\left(1+\sqrt{1-z}\right)^{2/3}
-\frac{z^{1/3}}{2^{5/3}3^{1/2}\pi}\,
2^{2/3}\left(1+\sqrt{1-z}\right)^{-2/3}\]
\[
=\frac{z^{-1/3}}{2\cdot 3^{1/2}\pi}\,
\left(1+\sqrt{1-z}\right)^{2/3}
-\frac{z^{1/3}}{2\cdot 3^{1/2}\pi}\,
\left(\frac{1-\sqrt{1-z}}{z}\right)^{2/3}\]
\[
=\frac{z^{-1/3}}{2\cdot 3^{1/2}\pi}\,
\left(1+\sqrt{1-z}\right)^{2/3}
-\frac{z^{-1/3}}{2\cdot 3^{1/2}\pi}\,
\left(1-\sqrt{1-z}\right)^{2/3}
\]
and this yields (\ref{partpoltorapol}).

For $r=1$ we use (\ref{parthyperg2}):
\[
W_{3/2,1}(x)=\frac{z^{-1/6}}{2^{2/3}\pi}\,
{}_{2}F_{1}\!\left(\left.
\frac{5}{6},\frac{-2}{3};\,\frac{2}{3}\right|z\right)
+\frac{z^{1/6}}{2^{1/3}\pi}\,
{}_{2}F_{1}\!\left(\left.
\frac{7}{6},\frac{-1}{3};\,\frac{4}{3}\right|z\right)
\]
\[
=\frac{z^{-1/6}}{4\pi}\left(1+\sqrt{1-z}\right)^{1/3}\left(3\sqrt{1-z}-1\right)
+\frac{z^{1/6}}{4\pi}\left(1+\sqrt{1-z}\right)^{-1/3}\left(3\sqrt{1-z}+1\right)
\]
\[
=\frac{z^{-1/6}}{4\pi}\left(1+\sqrt{1-z}\right)^{1/3}\left(3\sqrt{1-z}-1\right)
+\frac{z^{-1/6}}{4\pi}\left(1-\sqrt{1-z}\right)^{1/3}\left(3\sqrt{1-z}+1\right).
\]
Now we have
\[
\left(1+\sqrt{1-z}\right)^{1/3}\left(3\sqrt{1-z}-1\right)
=-\left(1+\sqrt{1-z}\right)^{1/3}\left(3-3\sqrt{1-z}-2\right)
\]
\[
=-3z^{1/3}\left(1-\sqrt{1-z}\right)^{2/3}+2\left(1+\sqrt{1-z}\right)^{1/3}
\]
and similarly
\[
\left(1-\sqrt{1-z}\right)^{1/3}\left(3\sqrt{1-z}+1\right)
=3z^{1/3}\left(1+\sqrt{1-z}\right)^{2/3}-2\left(1-\sqrt{1-z}\right)^{1/3}.
\]
Therefore
\[
W_{3/2,1}(x)=
\frac{z^{-1/6}}{2\pi}\left(\left(1+\sqrt{1-z}\right)^{1/3}-\left(1-\sqrt{1-z}\right)^{1/3}\right)
\]
\[
+\frac{3z^{1/6}}{4\pi}\left(\left(1+\sqrt{1-z}\right)^{2/3}-\left(1-\sqrt{1-z}\right)^{2/3}\right),
\]
which entails (\ref{partpoltorajeden}).
The last statement is a consequence of (\ref{convpp}).
\end{proof}

The dilation $\mathbf{D}_{2}\mu(3/2,1/2)$,
with the density $W_{3/2,1/2}(x/2)/2$, is known as the \textit{Bures distribution},
see (4.4) in \cite{sz2004}. Moreover, the integer sequence $4^m A_m(3/2,1/2)$,
moments of $\mathbf{D}_{4}\mu(3/2,1/2)$, appears as A078531 in \cite{oeis}
and counts the number of symmetric noncrossing connected graphs on $2n+1$
equidistant nodes on a circle. The axis of symmetry is a diameter
of a circle passing through a given node, see \cite{flajoletnoy}.

The measure $\mu(3/2,1)$ is equal to $\mu(2,1)^{\boxtimes 1/2}$,
the multiplicative free square root of the Marchenko-Pastur distribution.

For the sake of completeness we also include the cases $p=3,r=1$
and $p=3,r=2$, which have already appeared in \cite{ps,pezy}.

\begin{theorem}\label{thmtrzy}
Assume that $p=3$. Then for $r=1$ we have
\begin{equation}\label{parttrzyjeden}
W_{3,1}(x)=\frac{3\left(1+\sqrt{1-4x/27}\right)^{2/3}-2^{2/3} x^{1/3}}
{2^{4/3}3^{1/2}\pi x^{2/3}\left(1+\sqrt{1-4x/27}\right)^{1/3}},
\end{equation}
for $r=2$:
\begin{equation}\label{parttrzydwa}
W_{3,2}(x)=\frac{9\left(1+\sqrt{1-4x/27}\right)^{4/3}-2^{4/3}x^{2/3}}
{2^{5/3} 3^{3/2}\pi x^{1/3}\left(1+\sqrt{1-4x/27}\right)^{2/3}}
\end{equation}
and, finally, $W_{3,3}(x)=x\cdot W_{3,1}(x)$,
with $x\in(0,27/4)$.
\end{theorem}

\begin{proof}
For arbitrary $r$ we have
\begin{align*}
W_{3,r}(x)=\frac{2^{(6-2r)/3}\sin\big(\pi r/3\big)}{3^{3-r}\pi}\,
{}_{3}F_{2}\!\left(\left.
\frac{r}{3},\frac{3-r}{6},\frac{-r}{6};\,\frac{2}{3},\frac{1}{3}
\right|z\right)&z^{(r-3)/3}\\
-\frac{2^{(4-2r)/3}r\sin\big((1-2r)\pi/3\big)}{3^{3-r}\pi}\,
{}_{3}F_{2}\!\left(\left.
\frac{1+r}{3},\frac{5-r}{6},\frac{2-r}{6};\,\frac{4}{3},\frac{2}{3}
\right|z\right)&z^{(r-2)/3}\\
+\frac{r(r-1)\sin\big((1-r)\pi/3\big)}{2^{(1+2r)/3}3^{3-r}\pi}\,
{}_{3}F_{2}\!\left(\left.
\frac{2+r}{3},\frac{7-r}{6},\frac{4-r}{6};\,\frac{5}{3},\frac{4}{3}
\right|z\right)&z^{(r-1)/3},
\end{align*}
where $z=4x/27$. For $r=1$ and $r=2$ we have similar reduction
as in the previous proof. Here we will be using only (\ref{parthyperg1}).

Taking $r=1$ we get
\[
W_{3,1}(x)=\frac{2^{1/3}z^{-2/3}}{3^{3/2}\pi}\,
{}_{2}F_{1}\!\left(\left.
\frac{1}{3},\frac{-1}{6};\,\frac{2}{3}\right|z\right)
-\frac{z^{-1/3}}{2^{1/3}3^{3/2}\pi}\,
{}_{2}F_{1}\!\left(\left.
\frac{2}{3},\frac{1}{6};\,\frac{4}{3}\right|z\right)
\]
\[
=\frac{z^{-2/3}}{3^{3/2}\pi}\left(1+\sqrt{1-z}\right)^{1/3}
-\frac{z^{-1/3}}{3^{3/2}\pi}\left(1+\sqrt{1-z}\right)^{-1/3}
\]
\[
=\frac{\left(1+\sqrt{1-z}\right)^{2/3}-z^{1/3}}
{3^{3/2}\pi z^{2/3}\left(1+\sqrt{1-z}\right)^{1/3}},
\]
which implies (\ref{parttrzyjeden}).

Now we take $r=2$:
\[
W_{3,2}(x)=\frac{z^{-1/3}}{2^{1/3}3^{1/2}\pi}\,
{}_{2}F_{1}\!\left(\left.
\frac{1}{6},\frac{-1}{3};\,\frac{1}{3}\right|z\right)
-\frac{z^{1/3}}{2^{5/3}3^{1/2}\pi}\,
{}_{2}F_{1}\!\left(\left.
\frac{5}{6},\frac{1}{3};\,\frac{5}{3}\right|z\right)
\]
\[
=\frac{z^{-1/3}}{2\cdot 3^{1/2}\pi}\left(1+\sqrt{1-z}\right)^{2/3}
-\frac{z^{1/3}}{2\cdot 3^{1/2}\pi}\left(1+\sqrt{1-z}\right)^{-2/3}
\]
\[
=\frac{\left(1+\sqrt{1-z}\right)^{4/3}-z^{2/3}}
{2\cdot 3^{1/2}\pi z^{1/3}\left(1+\sqrt{1-z}\right)^{2/3}},
\]
and this gives us (\ref{parttrzydwa}).
Finally we apply (\ref{convpp}).
\end{proof}

Note that the measure $\mu(3,1)$ is equal to $\mu(2,1)^{\boxtimes 2}$,
the multiplicative free square of the Marchenko-Pastur distribution.

\section{Graphical representation of selected cases}

The explicit form of $W_{p,r}(x)$ given in Theorem~\ref{thmslater}
permits a graphical visualization for any rational $p>0$ and arbitrary $r>0$.
We shall represent some selected cases in Figures 1--9.
These graphs which are partly negative are drawn as dashed curves.
In Fig.~\ref{figura1} the graphs of the functions $W_{3/2,r}(x)$ for values
of $r$ ranging from $0.9$ to $2.3$ are given. For $r\le3/2$ these functions are positive,
otherwise they develop a negative part.
In Fig.~2 we represent $W_{5/2,r}(x)$ for $r$ ranging from $2$ to $3.4$.
In Fig.~3 we display the densities $W_{p,p}(x)$ for $p=6/5,5/4,4/3$ and $3/2$.
All these densities are unimodal and vanish at the extremities of their supports.
They can be therefore considered as generalizations of the Wigner's semicircle
distribution $W_{2,2}(x)$, see equation~(\ref{convsemicircle}).
In Fig.~4 we depict the functions $W_{4/3,r}(x)$, for values $r$ ranging from
$0.8$ to $2.4$. Here for $r\ge1.4$ negative contributions clearly appear.
In Fig.~5 and 6 we present six densities expressible through elementary functions,
namely $W_{3/2,r}(x)$ for $r=1/2,1,3/2$, see Theorem~\ref{thmtrzydrugie}
and $W_{3,r}(x)$ for $r=1,2,3$, see Theorem~\ref{thmtrzy}.
In Fig.~7 the set of densities $W_{p,1}(x)$ for five fractional
values of $p$ is presented. The appearance of maximum near $x=1$
corresponds to the fact that $\mu(p,1)$ weakly converges to $\delta_1$
as $p\to1^{+}$. In Fig.~8 the fine details of densities $W_{p,1}(x)$
for $p=5/2,7/3,9/4,11/5$, on a narrower range $2\le x\le4.5$ are presented.
In Fig.~9 we display the densities $W_{p,1}(x)$ for $p=2,5/2,3,7/2,4$,
near the upper edge of their respective supports, for $3\le x\le 9.5$.

The Fig.~10 summarizes our results in the $p>0,r>0$ quadrant of the
$(p,r)$ plane, describing the Raney numbers
(c.f. Fig.~5.1 in \cite{mlotkowski2010} and Fig.~7 in \cite{pezy}).
The shaded region $\Sigma$ indicates the probability measures $\mu(p,r)$
(i.e. where $W_{p,r}(x)$ is a nonegative function).
The vertical line $p=2$ and the stars indicate the pairs $(p,r)$ for which
$W_{p,r}(x)$ is an elementary function, see Corollary~\ref{cordwa},
Theorem~\ref{thmtrzydrugie} and Theorem~\ref{thmtrzy}.
The points $(3/2,1)$ and $(3,1)$ correspond to the multiplicative
free powers $\mathrm{MP}^{\boxtimes 1/2}$ and $\mathrm{MP}^{\boxtimes 2}$
of the Marchenko-Pastur distribution $\mathrm{MP}$.
Symbol $\mathrm{B}$ at $(3/2,1/2)$ indicates the Bures distribution and $\mathrm{SC}$
at $(2,2)$ denotes the semicircle law centered at $x=2$, with radius 2.

It is our pleasure to thank M.~Bo\.{z}ejko, Z.~Burda, K.~G\'{o}rska, I.~Nechita and M.~A.~Nowak
for fruitful interactions.

\begin{figure}
\caption{Raney distributions $W_{3/2,r}(x)$ with values of the parameter $r$
labeling each curve.  For $r>p$ solutions  drawn with dashed lines
are not positive.}
\centering
\includegraphics[width=0.6\textwidth]{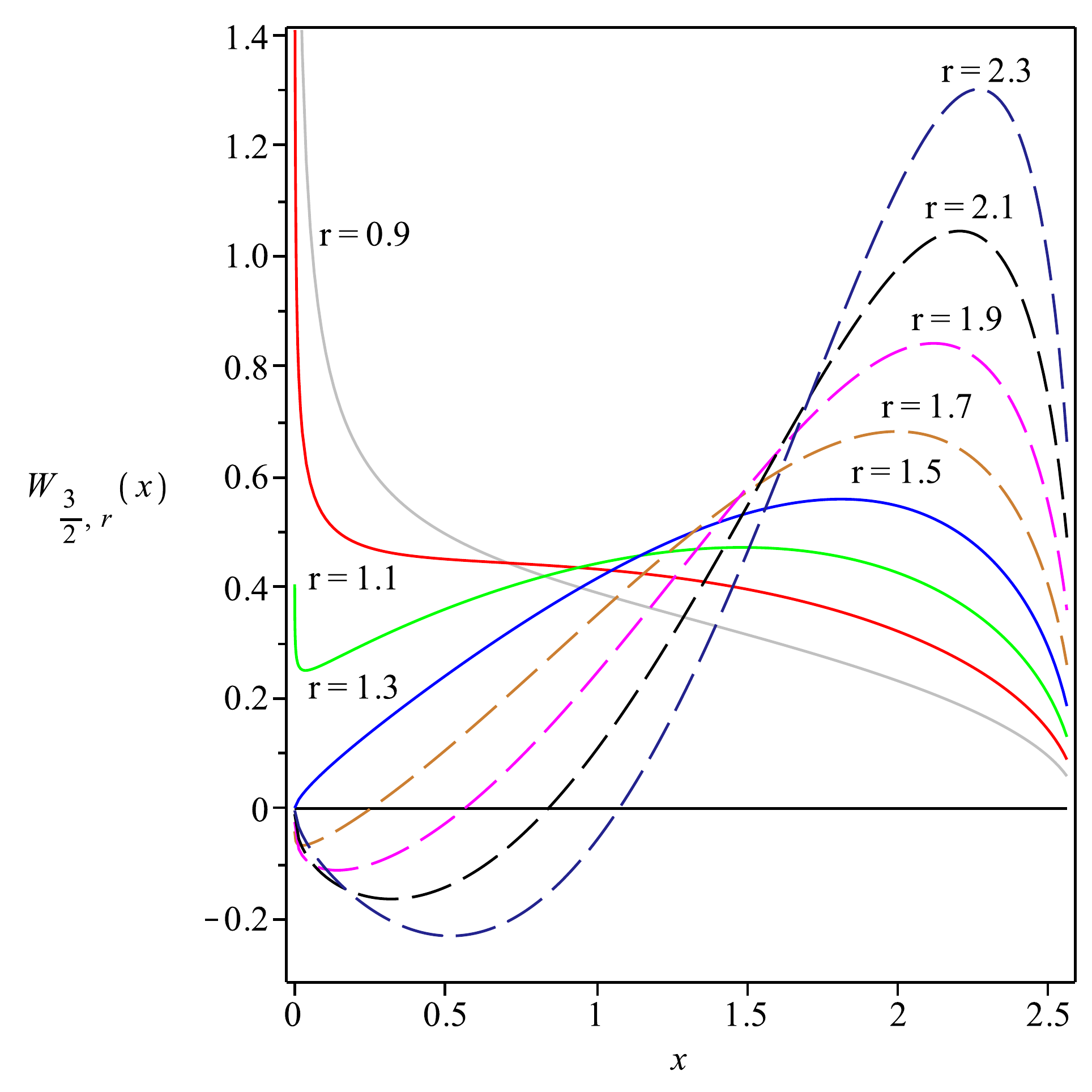}
%\caption{ccc}
\label{figura1}
\end{figure}

\begin{figure}
\caption{As in Fig.~\ref{figura1} for Raney distributions $W_{5/2,r}(x)$.}
\centering
\includegraphics[width=0.6\textwidth]{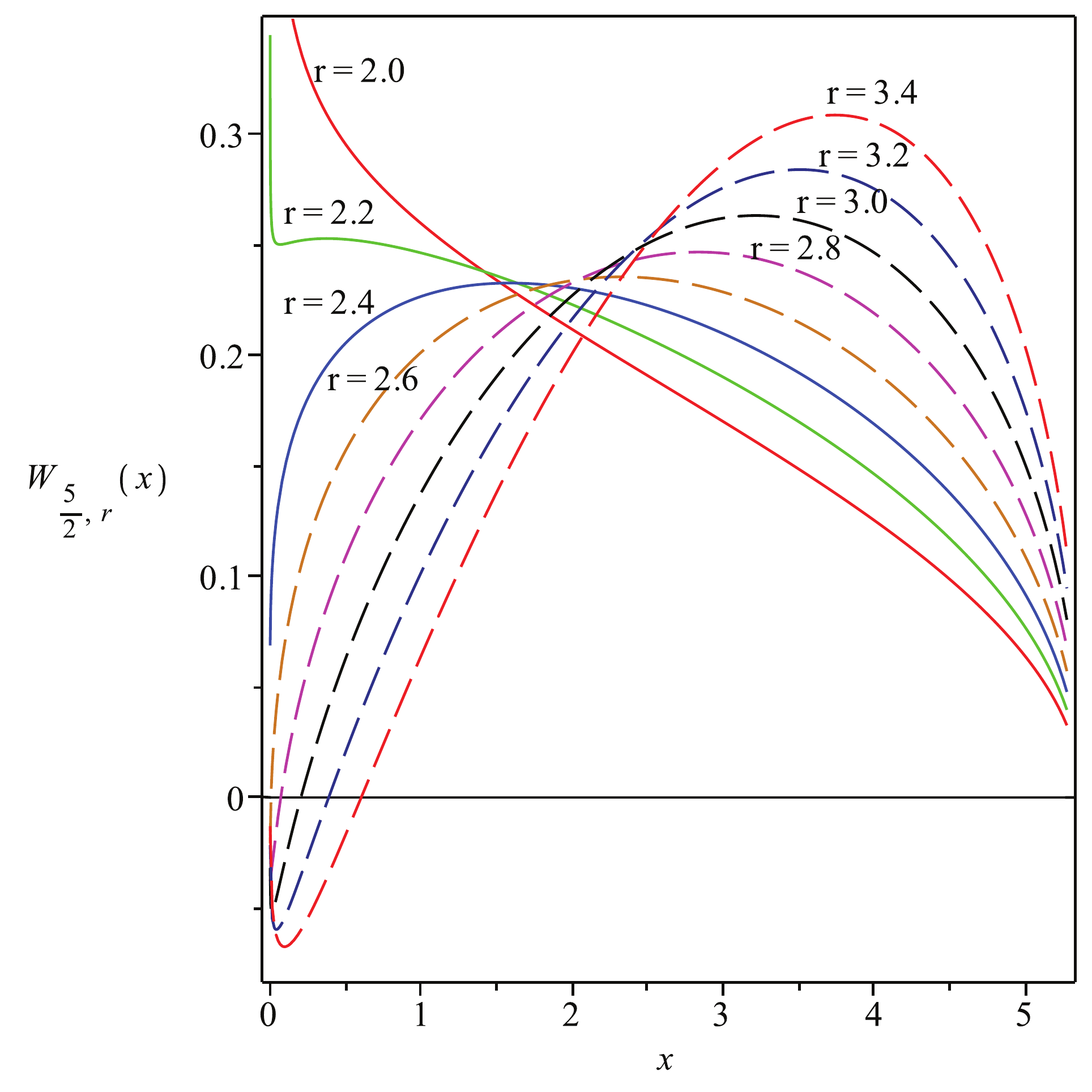}
%\caption{cc2}
\label{figura2}
\end{figure}

\begin{figure}
\caption{Diagonal Raney distributions $W_{p,p}(x)$ with values of the
parameter $p$ labeling each curve.}
\centering
\includegraphics[width=0.6\textwidth]{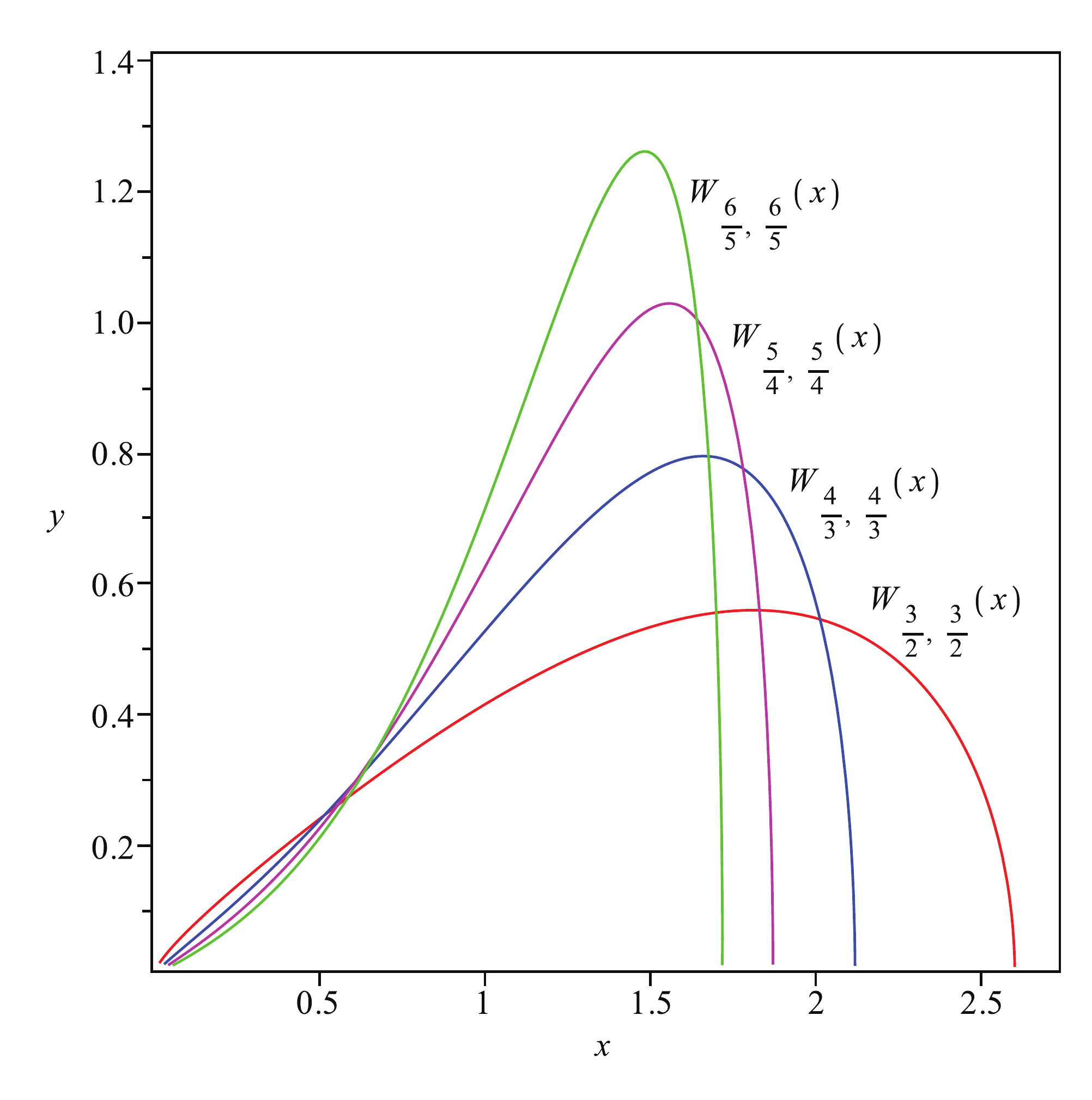}
%\caption{cc3}
\label{figura3}
\end{figure}

\begin{figure}
\caption{The functions $W_{4/3,r}(x)$ for $r$ ranging from
$0.8$ to $2.4$.}
\centering
\includegraphics[width=0.6\textwidth]{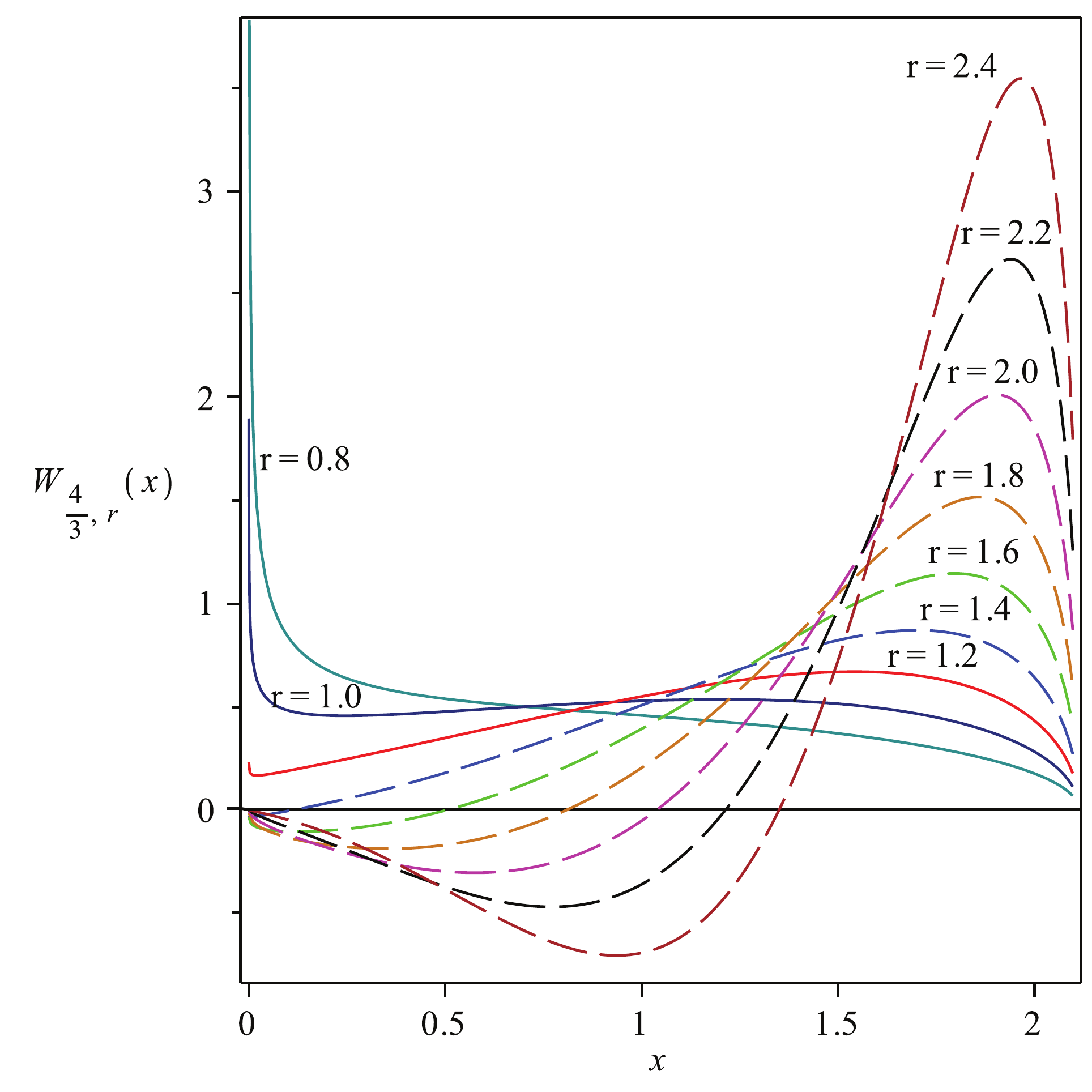}
%\caption{cc4}
\label{figura4}
\end{figure}

\begin{figure}
\caption{ Raney distributions $W_{3/2,r}(x)$ with values of the
parameter $r$ labeling each curve. The case $W_{3/2,1}(x)$ represents
the multiplicative free square root of the Marchenko Pastur distribution.}
\centering
\includegraphics[width=0.6\textwidth]{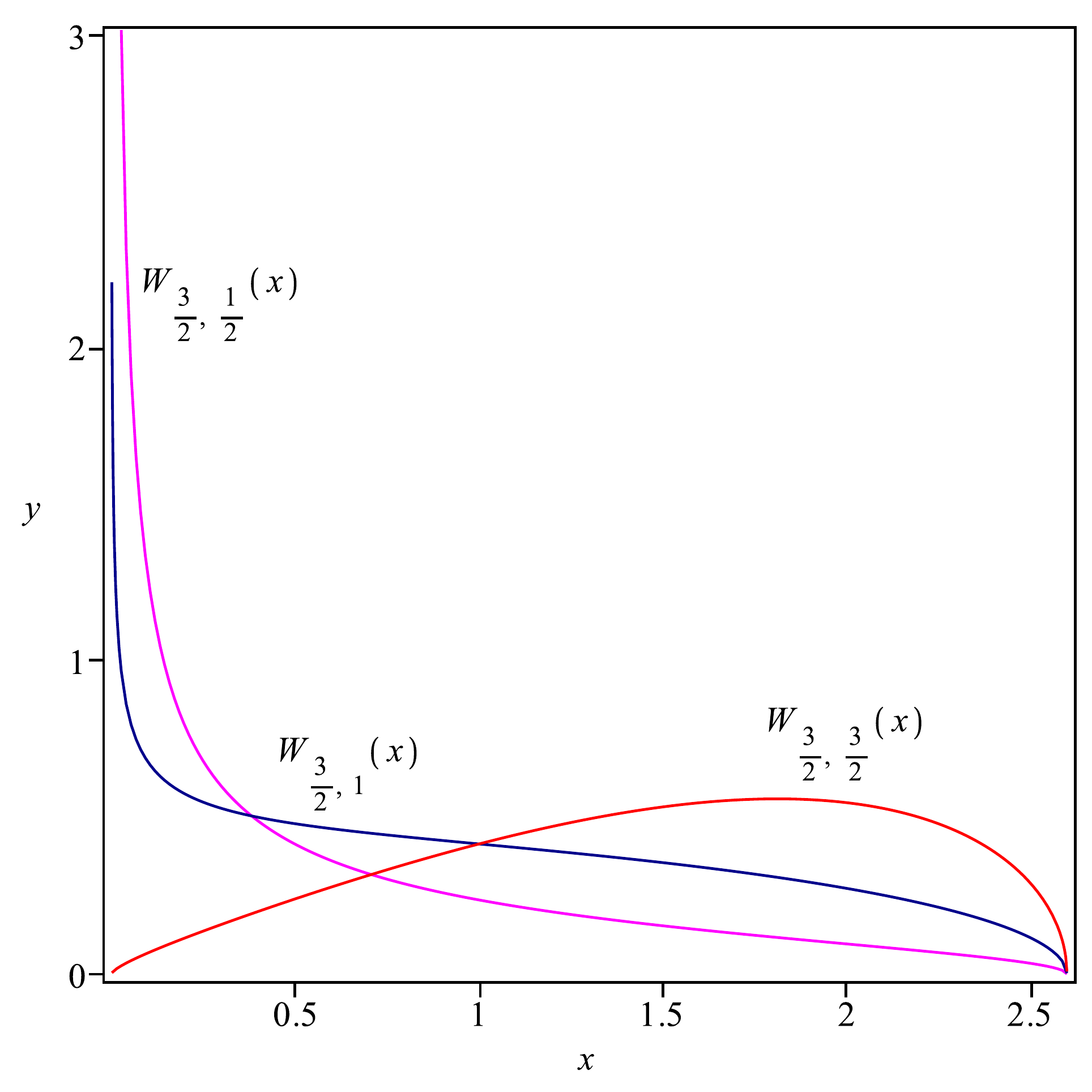}
%\caption{cc5}
\label{figura5}
\end{figure}

\begin{figure}
\caption{Raney distributions $W_{3,r}(x)$ with values of the parameter
$r$ labeling each curve. The case $W_{3,1}(x)$
represents the  multiplicative free square of the Marchenko Pastur distribution.}
\centering
\includegraphics[width=0.6\textwidth]{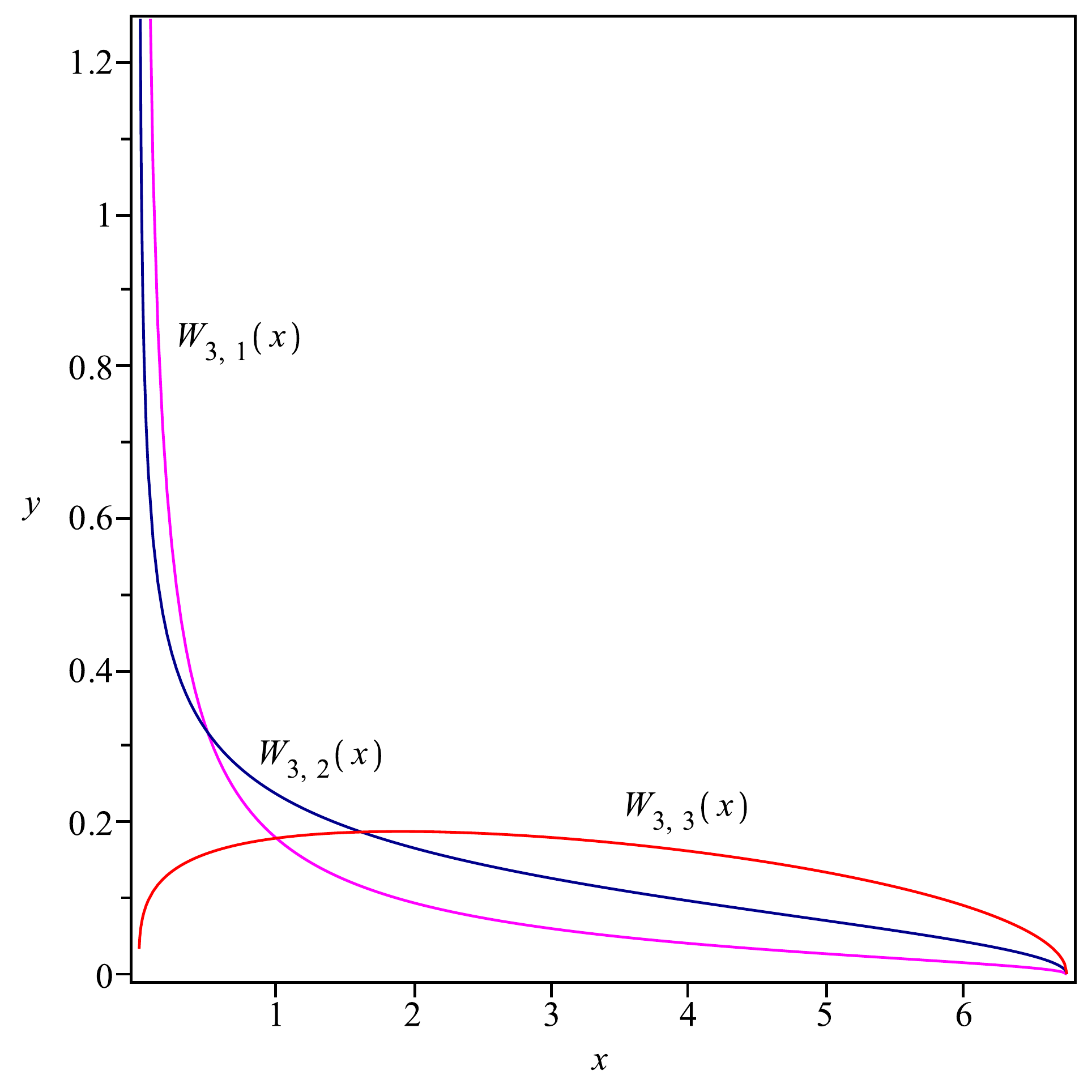}
%\caption{cc6}
\label{figura6}
\end{figure}

\begin{figure}
\caption{Raney distributions $W_{p,1}(x)$ with values of the parameter
$p$ labeling each curve.
The case $W_{3/2,1}(x)$ represents
the multiplicative free square root of the Marchenko--Pastur distribution,
$MP^{\boxtimes 1/2}$.}
\centering
\includegraphics[width=0.6\textwidth]{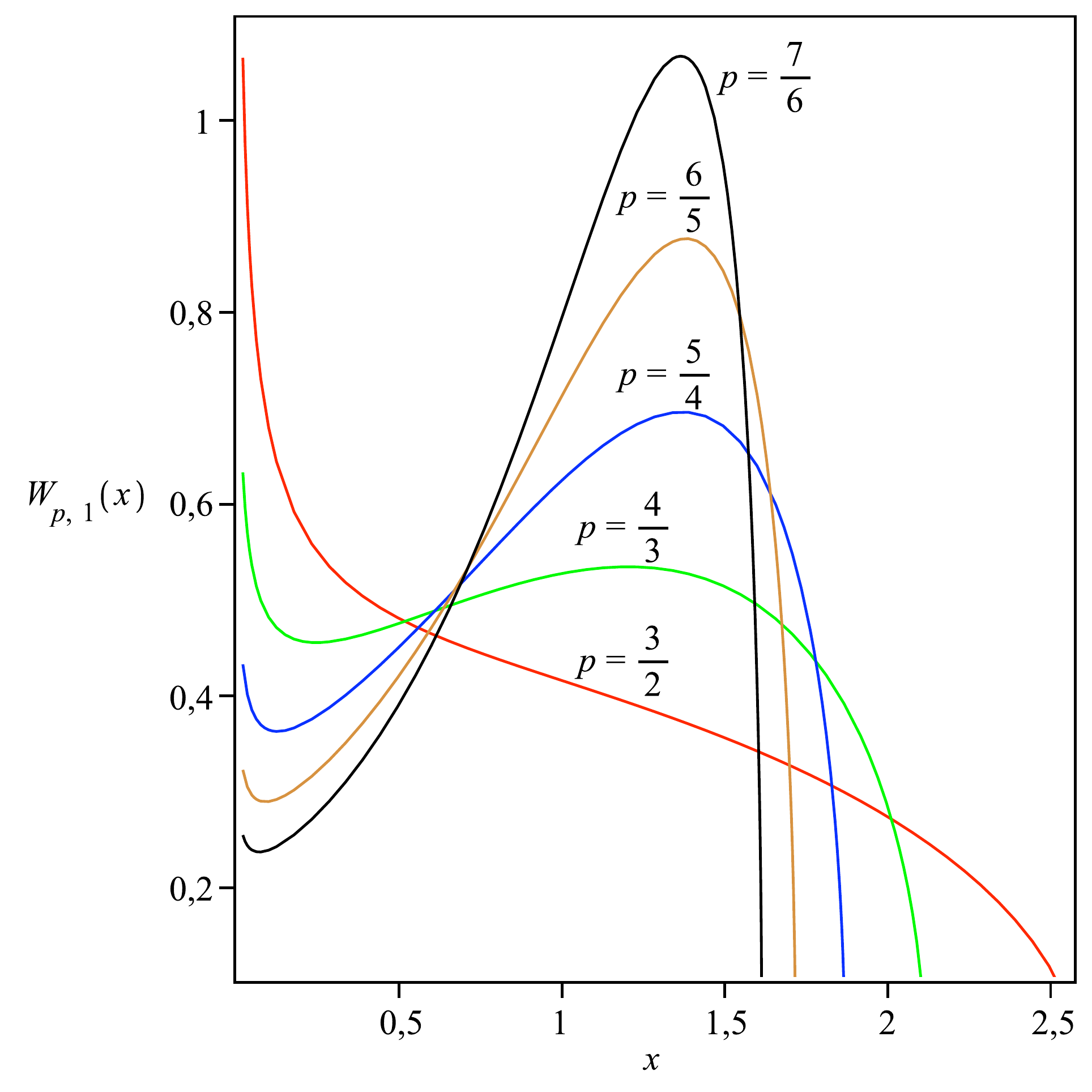}
%\caption{ccc}
\label{figura7}
\end{figure}

\begin{figure}
\caption{Tails of the Raney distributions $W_{p,1}(x)$ with values of
the parameter $p$ labeling each curve.}
\centering
\includegraphics[width=0.6\textwidth]{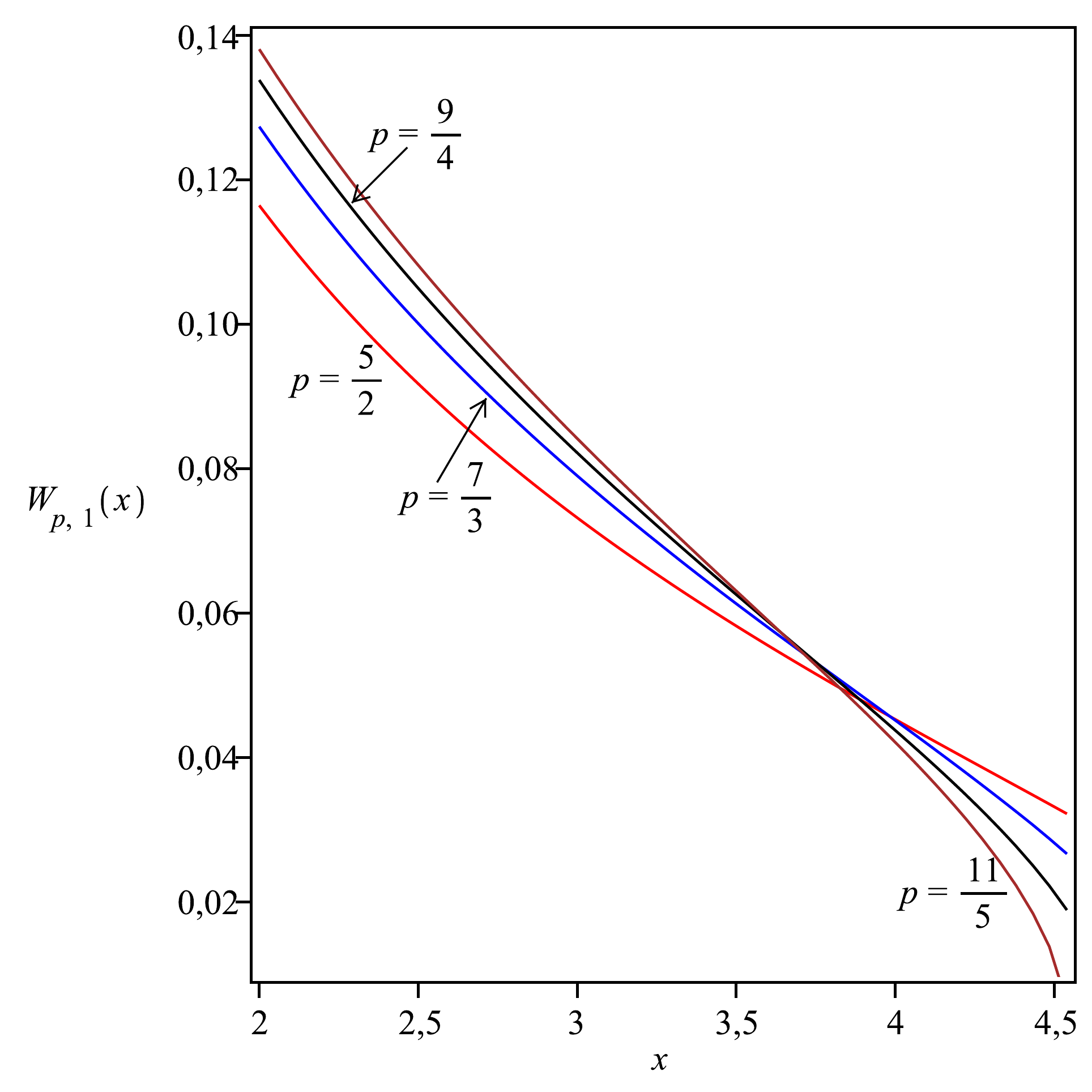}
%\caption{ccc}
\label{figura8}
\end{figure}

\begin{figure}
\caption{As in Fig. \ref{figura8} for larger values of the parameter $p$.}
\centering
\includegraphics[width=0.6\textwidth]{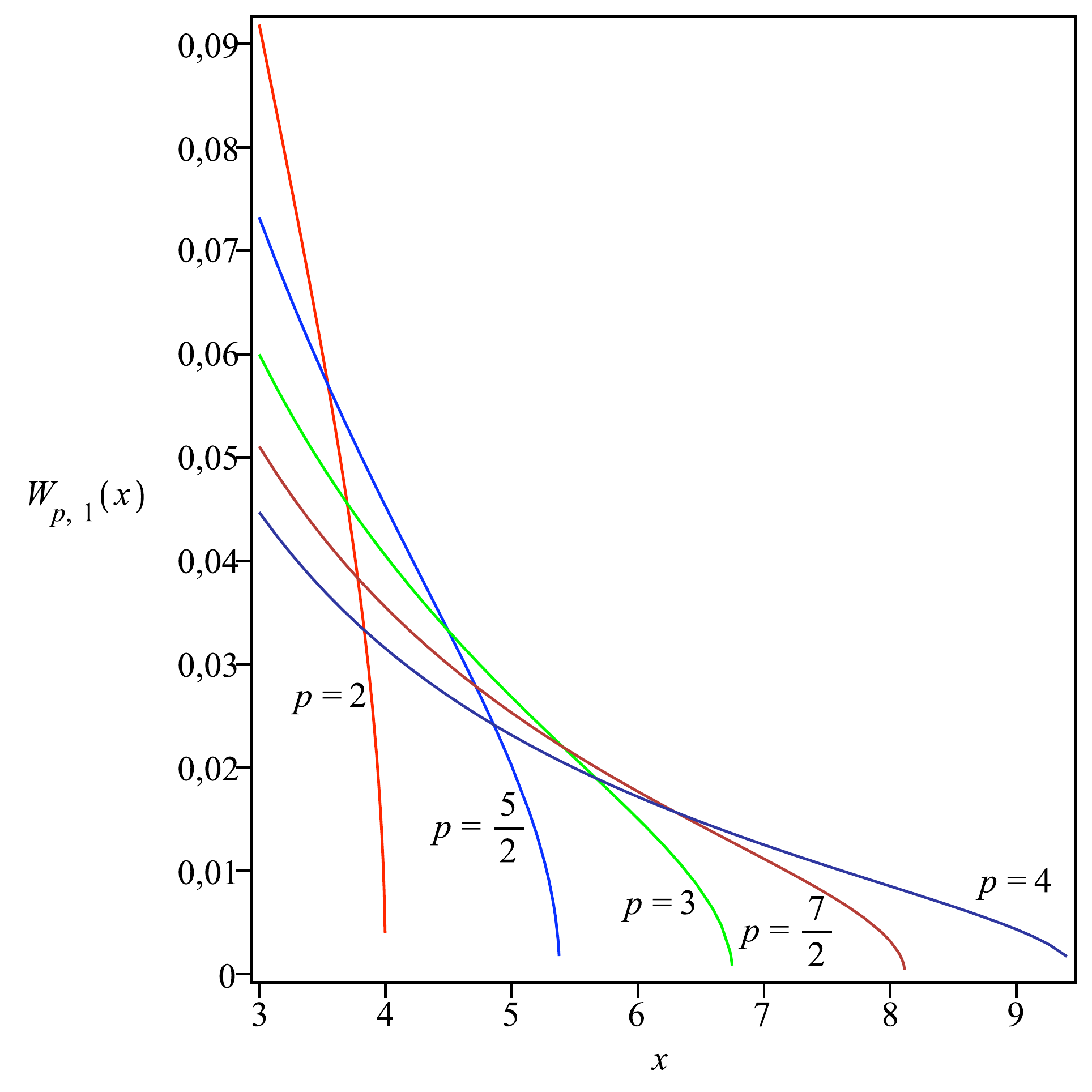}
%\caption{ccc}
\label{figura9}
\end{figure}

\begin{figure}
\caption{Parameter plane $(p,r)$ describing the Raney numbers.
The shaded set  $\Sigma$ corresponds  to nonnegative probability
measures $\mu(p,r)$. The vertical line $p=2$ and the stars represent values
of parameters for which $W_{p,r}(x)$ is an elementary function.
Here $\mathrm{MP}$ denotes the Marchenko--Pastur distribution,
$\textrm{MP}^{\boxtimes s}$ its $s$-th free mutiplicative power,  $\mathrm{B}$-the
Bures distribution while $\mathrm{SC}$ denotes the semicircle law.
For $p>1$ the points $(p,p)$ on the upper edge of $\Sigma$ represent the generalizations
of the Wigner semicircle law, see Fig.~\ref{figura3}.}
\centering
\includegraphics[width=0.6\textwidth]{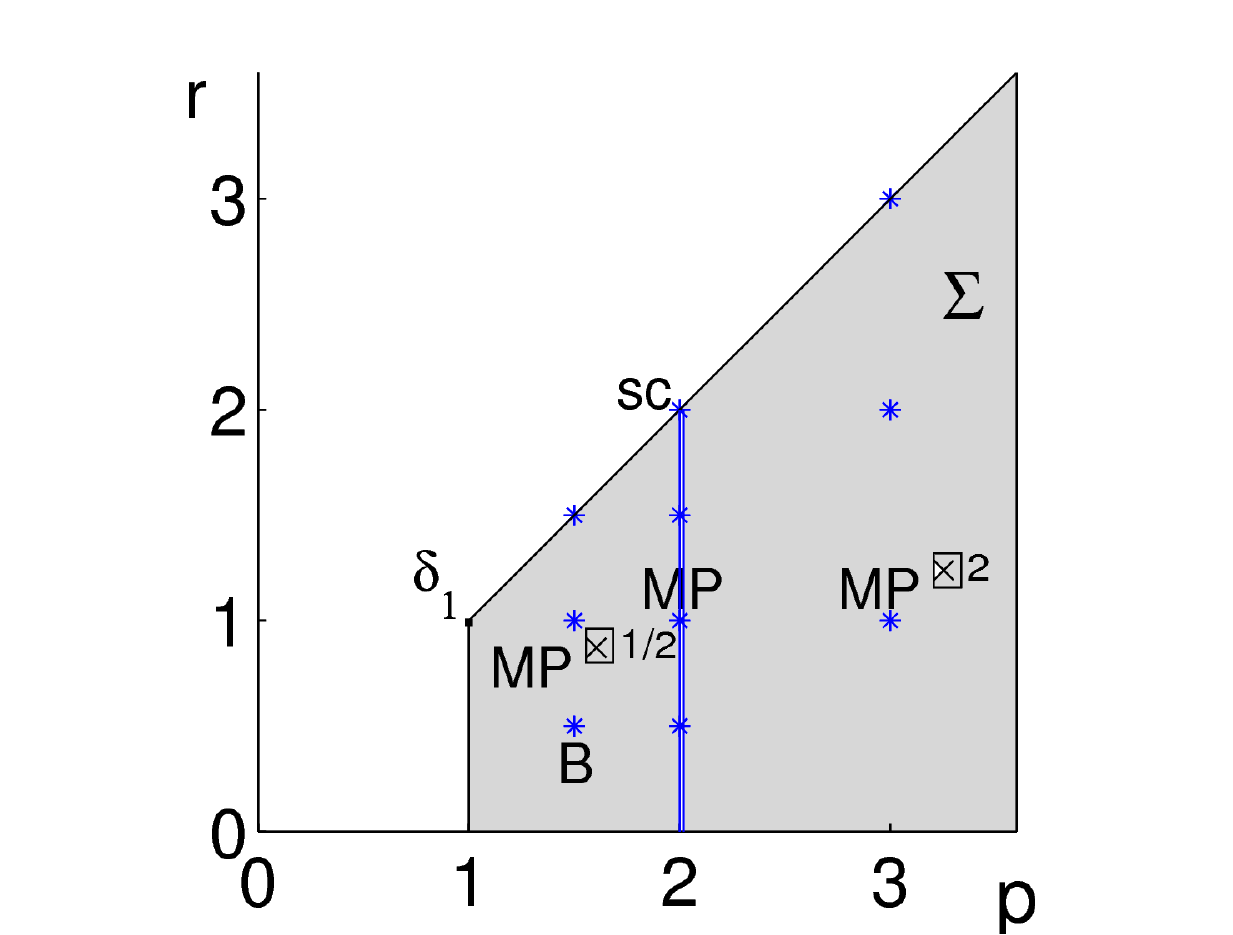}
%\caption{ccc}
\label{mapa2}
\end{figure}

\end{document}